﻿
\documentclass[12pt]{amsart}
﻿
﻿
﻿
\makeatletter \@addtoreset{equation}{section}

\makeatother \makeatletter
\newtheorem{theorem}{Theorem}[section]
\newtheorem{corollary}[theorem]{Corollary}
\newtheorem{definition}[theorem]{Definition}
\newtheorem{lemma}[theorem]{Lemma}
\newtheorem{proposition}[theorem]{Proposition}
\newtheorem{remark}[theorem]{Remark}
\newtheorem{example}[theorem]{Example}
\newtheorem*{theorem*}{Theorem}
\newtheorem*{proposition*}{Proposition}﻿

\def\pr1{\prod\hskip -2.07ex * \hskip 0.9 ex}

\newcommand{\aaa}{ \textbf{\em a}}

﻿
\usepackage{latexsym} \usepackage{amsmath} \usepackage{amssymb}
\usepackage{amsthm}
\usepackage{centernot}
\usepackage{cancel}\usepackage{color}
﻿
\begin{document}

\noindent
\title{On the irreducibility of slice algebraic sets}
\author{Anna Gori $^1$} 
 \thanks{$^1$ Dipartimento di Matematica - Universit\`a di Milano,                
 Via Saldini 50, 20133  Milano, Italy}
 \author{Giulia Sarfatti $^2$}
 \thanks{ $^2$ DIISM - Universit\`a Politecnica delle Marche,               
 	Via Brecce Bianche 12,  60131, Ancona, Italy}
 \author{Fabio Vlacci $^3$}
 \thanks{$^3$ MIGe  - Universit\`a di Trieste, Piazzale Europa 1, 34100, Trieste, Italy}
 \begin{abstract}
{
In the present paper we investigate the {relations} between {irreducible slice algebraic} sets in $\mathbb{H}^n$ and {quasi prime} right ideals of the ring of slice regular polynomials in $n$ quaternionic variables. We
 provide  algebraic conditions 
 on right ideals of slice regular polynomials {which guarantee the}
 irreducibility of the corresponding {slice} algebraic sets
 and show that radical ideals associated with irreducible slice algebraic sets are quasi prime.
 Furthermore we establish that this correspondence is an equivalence 
 in the case of principal right ideals.}

\end{abstract}﻿

\maketitle

\section*{Acknowledgments}
The authors are partially supported by:
GNSAGA-INdAM via the project ``Hypercomplex function theory and
applications''; the first author is also partially supported by MUR
project PRIN 2022 ``Real and Complex Manifolds: Geometry and
Holomorphic Dynamics'', the second and the third authors are also
partially supported by MUR projects PRIN 2022 ``Interactions between
Geometric Structures and Function Theories''.

\section{Introduction}
﻿
﻿
\noindent In Algebraic Geometry ideals of polynomials  and algebraic varieties are deeply related and 
this correspondence often refers to the Hilbert Nullstellensatz.

\noindent {In the quaternionic setting, a version of the so called {\em Strong} Hilbert Nullstellensatz was proven, independently, in \cite{israeliani-3} and in \cite{Nul2}. The rings of polynomials taken into account in these two papers are different but isomorphic. Let $\mathbb H$ denote the skew-field of quaternions. In the present paper we will consider, as in \cite{Nul2}, 
polynomial functions with quaternionic coefficients defined on $\mathbb H^n$ as 
\[P(q_1,\ldots, q_n)=
\sum_{ \substack{\ell_1=0,\ldots, L_1  \\
		\cdots\\ \ell_n=0,\ldots, L_n} }{q_1}^{\ell_1}\cdots {q_n}^{\ell_n}a_{\ell_1,\ldots,\ell_n}, \]
which belong to the class of {\em slice regular functions} according to Definition 2.5 in \cite{severalvariables}.
The set $\mathbb{H}[q_1,\ldots,q_n]$ of slice regular polynomials in $n$ quaternionic variables 
can be equipped with a suitable
notion of (non-commutative) product, called {\em slice product} and denoted by the symbol $*$, endowing  $(\mathbb{H}[q_1,\ldots,q_n], +,*)$ with the structure of non-commutative ring. 
The set of slice regular polynomials vanishing on a given subset $Z$ of $\mathbb{H}^n$ is not in general an ideal of
$\mathbb{H}[q_1,\ldots,q_n]$. Therefore we  define  $\mathcal{J}(Z)$ to be the right ideal 
generated by
slice regular polynomials in several variables which vanish on the subset $Z$. The choice of working with right ideals is due to the definition of slice regularity we refer to (see \cite{severalvariables}).
On the other hand, starting from a right ideal $I$ of $\mathbb{H}[q_1,\ldots,q_n]$, we 
define $\mathcal{V}(I)$ to be the set of common zeros of slice regular polynomials
in $I$ and
 $\mathcal V_c(I)$ to be the subset of $\mathcal V(I)$ whose points have commuting components. If $\mathbb S$ denotes the two-dimensional sphere of imaginary units in $\mathbb H$, namely $\mathbb{S}=\{J\in\mathbb{H}\ J^2=-1\}$, then
$\mathcal{V}_c(I)=\mathcal{V}(I)\cap\bigcup_{K\in \mathbb{S}}\mathbb{C}_K^n$,
where $\mathbb{C}_K:=\mathbb{R}+\mathbb{R} K$ can be
identified with the complex plane $\mathbb{C}$ { for any $K \in \mathbb S$}. 
 \\
\noindent Introducing, as in \cite{israeliani-1}, the radical  $\sqrt{I}$ of a right ideal $I$ as the intersection
of    all {\em completely prime} right ideals (see Definition \ref{cprime}) that contain $I$, 
the Strong Hilbert Nullstellensatz in $\mathbb{H}^n$ reads as follows
\begin{theorem*}[{\cite[Theorem 4.3]{Nul2}}]
	Let $I$ be a right ideal in $\mathbb{H}[q_1,\ldots, q_n]$. Then
	\[\mathcal{J}(\mathcal{V}(I))=\sqrt{I}.\] 	
\end{theorem*}
\noindent Motivated by this connection between ideals of polynomials and subsets of $\mathbb H^n$, in \cite{Nul2} we introduced the notion of {\em slice algebraic
	sets} which defines a Zariski-type topology on $\mathbb H^n$ and we proved that every subset of the form
$\mathcal{V}(I)$ or $\mathcal V_c(I)$ is a slice algebraic set.}

The aim of the present paper is to investigate the relations between the geometric notion of {\em irreducibility},  namely the fact that 
{$\mathcal{V}(I)$ (or $\mathcal V_c(I)$) cannot be obtained as the union of two different slice algebraic sets},
 with the algebraic properties of the corresponding right ideal $I$ of {slice regular} polynomials. 

﻿
﻿In the complex setting, it is a well-known fact that an algebraic variety is irreducible if and only if its associated ideal is {\em prime}.  
In view of the statement of the Strong Nullstellensatz, a first guess would be to replace prime ideals with completely prime ideals to characterize irreducible slice algebraic sets in $\mathbb{H}^n$. 
%
However, this translation does not work:
 there are irreducible slice algebraic sets  whose corresponding ideals are not completely prime. 
 As an example consider the set $\mathbb{S}$ which is the zero locus 
 of slice regular polynomials belonging to the right ideal generated by $q^2+1$ 
 in $\mathbb H[q]$. While $\mathbb{S}$ is irreducible, 
the right ideal of $\mathbb{H}[q]$ generated by $q^2+1=(q+i)*(q-i)=:(q+i)^s$ is not {completely prime}. 
%
%
%
%
%
%
﻿
{In order to fill the gap in the correspondence between completely prime ideals and irreducible slice algebraic sets}, we introduce the new and wider class of {\em quasi prime} right ideals (see Definition \ref{qpideal}).
﻿This notion is linked to the geometry of the zero sets of slice regular polynomials and turns out to be the proper tool to
characterize the irreducibility of algebraic sets. In fact, 
in one direction we prove 
﻿\begin{theorem*} Let $I$ be a radical right ideal of $(\mathbb{H}[q_1,\ldots,q_n], +, *)$ such that $\mathcal{V}_c(I)$ is irreducible, then $I$ is a quasi prime ideal.
 \end{theorem*}  
\noindent To provide a converse of the previous result, starting from a right ideal $I$, we introduce 
the {\em symmetrized ideal} $\mathcal{S}(I)$ as the right ideal generated by the symmetrizations $P^s$ (see Definition \ref{R-coniugata2}) of slice regular polynomials $P$ belonging to the ideal $I$.  \\
Given a set $Z_c\subseteq \mathbb H^n$ of points with commuting components, we define its {\em symmetrization} as $\mathbb{S}_{Z_c}:=\bigcup_{\textbf{{\em a}} \in Z_c}\mathbb S_{\textbf {\em a}}$, where
	 $\mathbb S_{\textbf {\em a}}=\{ (ga_1g^{-1},\ldots, ga_ng^{-1}) \ : \ g\in\mathbb{H}\setminus\{0\}\}$  if $\textbf {\em a}=(a_1,\ldots, a_n)$.
{	Symmetrized ideals and symmetrized subsets are related by the following result.
	\begin{theorem*}
		Let $I$ be a right ideal of $\mathbb H[q_1,\ldots,q_n]$.  Then the symmetrization of $\mathcal V_c(I)$ is a slice algebraic set, more precisely
		$	\mathbb{S}_{\mathcal{V}_c(I)}=	\mathcal{V}_c( \mathcal{S}(I)).$
			\end{theorem*}
	}
\noindent Using this relevant fact, we are able to prove the following 

\begin{theorem*}
	Let $I$ be a {quasi} prime {and radical} right ideal of $\mathbb H[q_1,\ldots,q_n]$. If the radical of $\mathcal S(I)$ is quasi prime, then $\mathcal V_c(I)$ is irreducible.
\end{theorem*}

\noindent Furthermore, we prove that the irreducibility of $\mathcal{V}_c(I)$ implies the irreducibility of $\mathcal{V}(I)$ (see Proposition \ref{vc-v}) and provide examples that the converse is not true.

	\noindent We then consider the case of principal right ideals, namely those generated by a single polynomial. Using factorization properties of the generator, we completely characterize principal ideals whose corresponding algebraic sets are irreducible. 
	\begin{theorem*} A principal right ideal $I$ of $\mathbb H[q_1,\ldots,q_n]$ is quasi prime if and only if $I$ is radical
	 and $\mathcal{V}_c(I)$ is irreducible.
	\end{theorem*}
%
%
%

%
	
\noindent The paper is organized as follows: after recalling  in Section 2 some  background results concerning ideals of
$(\mathbb{H}[q_1,\ldots, q_n], +,*)$,
together with several preliminary results and definitions, in Section 3 we introduce the two notions of symmetrization of an ideal and of a slice algebraic set and we link these two concepts in Theorem \ref{simmalg}. In Section 4 we further investigate the relations between algebraic sets and ideals, and in Section 5
we provide necessary and sufficient conditions for the irreducibility of a slice algebraic set (see 
 Theorem \ref{VirredIprime}, Theorem \ref{nosfere}, Theorem \ref{cpirred} and  Corollary \ref{coro3}). 
In Section 6 we focus our attention on principal right ideals giving, in Theorem \ref{principal}, a complete characterization of right ideals $I$ whose corresponding slice algebraic sets $\mathcal V_c(I)$ are irreducible.

﻿
﻿
	
﻿
\section{Preliminaries}\label{prel}
\noindent Let $\mathbb{H}{=\mathbb{R}+i\mathbb{R}+j\mathbb{R}+k\mathbb{R}}$ denote the skew field of
quaternions and let $\mathbb{S}=\{J \in \mathbb{H} \ : \ J^2=-1\}$ be
the two dimensional sphere of quaternionic imaginary units.  Then
\[ \mathbb{H}=\bigcup_{J\in \mathbb{S}}(\mathbb{R}+\mathbb{R} J),  
\]
where, for any $J\in \mathbb{S}$, the ``slice'' $\mathbb{C}_J:=\mathbb{R}+\mathbb{R} J$ can be
identified with the complex plane $\mathbb{C}$.  { Hence, any $q\in \mathbb{H}$ can be
  written as $q=x+yJ$ with $x,y \in \mathbb{R}$ and $J \in
  \mathbb{S}$. The {\it real part }of $q$ is ${\rm Re}(q)=x$ and its
         {\it imaginary part} is ${\rm Im}(q)=yJ$; the {\it conjugate}
         of $q$ is $\bar q:={\rm Re}(q)-{\rm Im}(q)$.  For any non-real
         quaternion $a\in \mathbb{H}\setminus \mathbb{R}$ we will
         denote by $J_a:=\frac{{\rm Im}(a)}{|{\rm Im}(a)|}\in
         \mathbb{S}$ and by $\mathbb{S}_a:=\{{\rm Re}(a)+J |{\rm
           Im}(a)| \ : \ J\in \mathbb{S}\}$. If $a\in \mathbb{R}$,
         then $J_a$ is any imaginary unit.}
Similarly, if $\textbf{\em a}=(a_1,\ldots,a_n)\in \mathbb H^n$, then the
	 {\em spherical set} \[\mathbb S_{(a_1,\ldots,a_n)}=\mathbb S_{\textbf{\em a}}
	 :=\{(g^{-1}a_1g, \ldots, g^{-1}a_ng) : g \in \mathbb{H}\setminus\{0\} \}\]
is the set obtained by simultaneously rotating each coordinate of the point 
$\textbf{\em a}=(a_1,\ldots,a_n)$. If there exists $J\in \mathbb S$ such that $\textbf{\em a}=(a_1,\ldots,a_n) \in \mathbb C_J^n$, then $\mathbb S_{\textbf{\em a}}$ is called {\em arranged spherical set}.
﻿
Slice regular  polynomials in $n$ quaternionic variables are polynomial functions
$P:{\mathbb{H}^n}  \to \mathbb{H}$, of the form
	\[	 (q_1,\ldots, q_n)\mapsto P(q_1,\ldots, q_n)=
	\sum_{ \substack{\ell_1=0,\ldots, L_1  \\
	\cdots\\ \ell_n=0,\ldots, L_n} }{q_1}^{\ell_1}\cdots {q_n}^{\ell_n}a_{\ell_1,\ldots,\ell_n} \]
	with $a_{\ell_1,\ldots,\ell_n}\in\mathbb{H}$, 
where $\deg_{{q_{\ell_j}}}P:=L_j$.
﻿

Slice regular polynomial functions of several variables can be endowed with an
appropriate notion of product, the so called {\em slice product}, that
will be denoted by the symbol $*$.
Let us recall here how it works for
slice regular polynomials in two variables.
\begin{definition}
	If $P(q_1,q_2)=\sum_{ \substack{n=0,\ldots,N_1  \\ m=0,\ldots,N_2} } q_1^nq_2^m a_{n,m}$ and
        $Q(q)=\sum_{ \substack{n=0,\ldots,L_1  \\ m=0,\ldots,L_2} }q_1^nq_2^m b_{n,m}$ are two
        slice regular polynomials,
        then the {\em $*$-product} of $P$ and $Q$ is the slice regular polynomial
        defined by
$$
P*Q (q_1,q_2):=\sum_{ \substack{n=0,\ldots,N_1+L_1  \\ m=0,\ldots,N_2+L_2}}q_1^nq_2^m\sum_{ \substack{r=0,\ldots,n  \\ s=0,\ldots,m} }a_{r,s} b_{n-r,m-s}
$$
\end{definition}
﻿
﻿
%
﻿
\noindent Notice that the $*$ product is not commutative in general.
 In this way, $(\mathbb{H}[q_1,\ldots,q_n], +,*)$ has the structure of a non-commutative ring.
﻿
\noindent For points with commuting components, the $*$-product {has the following explicit expression.}
\begin{proposition}[{\cite[Proposition 3.1]{Nul2}}]\label{prodstar} 
	Let ${  \textbf{a}}=(a_1,\ldots,a_n)\in \mathbb H^n$ be
	such that 
	$a_la_m=a_ma_l$ for any $1\leq l,m\leq n$
	and let $P,Q \in \mathbb H[q_1,\ldots,q_n]$. Then
	\[P*Q(\textbf{a})=\left\{\begin{array}{lr}
		0 & \text {if}\ P(\textbf{a})=0\\
		P(\textbf{a})\cdot Q(P(\textbf{a})^{-1}a_1P(\textbf{a}), P(\textbf{a})^{-1}a_2P(\textbf{a}),\ldots,P(\textbf{a})^{-1}a_nP(\textbf{a})) & \text{if}\ P(\textbf{a})\neq0
	\end{array}\right.\] 	
\end{proposition}
﻿
﻿
{
﻿
﻿
\noindent As in the one-variable case, it is possible to introduce two
        operators on slice regular  functions. For the sake of
        simplicity, we define them only for polynomials of two
        quaternionic variables.
\begin{definition}\label{R-coniugata2}
	Let $P(q_1,q_2)=
	\sum_{ \substack{n=0,\ldots,N_1  \\ m=0,\ldots,N_2} } 
	q_1^nq_2^ma_{n,m}$ be
        a slice regular polynomial. Then the \emph{regular conjugate} of $P$ is the
        slice regular polynomial defined  by
	\[P^c(q_1,q_2)=\sum_{ \substack{n=0,\ldots,N_1  \\ m=0,\ldots,N_2} } 
	q_1^nq_2^m\overline{a_{n,m}},\]
	and the
	\emph{symmetrization} of $P$ is the slice regular polynomial defined  by 
	\[P^s= P*P^c = P^c*P.\]
\end{definition}
﻿
﻿
}
%
%
%
﻿
﻿
﻿
﻿

{ 
	\noindent  We
	recall the following version of Splitting Lemma which is the several variable counterpart of Lemma 1.3 in \cite{libroGSS}.
	\begin{lemma}\label{splitting}
		Let $P$
		be a slice regular polynomial in $n$ quaternionic variables. For   any $K\in \mathbb S$ and for any $L\in\mathbb S$ orthogonal to $K$ (with respect to the standard scalar product in $\mathbb{R}^3$), there exist two complex polynomials $F,G:\mathbb C_K^{n}\to \mathbb C_K^{n} $ such that for any $(z_1,\ldots,z_n)\in \mathbb C_K^n$
		\begin{equation}\label{split}
			P(z_1,\ldots,z_n)=F(z_1,\ldots,z_n)+G(z_1,\ldots,z_n)L
		\end{equation}
		and
	\begin{equation}\label{splitconj}
		P^c(z_1,\ldots,z_n)=\overline{F(\bar z_1,\ldots,\bar z_n)}-G(z_1,\ldots,z_n)L.
	\end{equation}	
	\end{lemma}
%
%
\begin{proof}	
Equation \eqref{split} is the statement of Lemma 2.1 in \cite{Nul2}.
	\noindent As a consequence, if $P(q_1,\ldots,q_n)=\sum_{\ell_1,\ldots,\ell_n}q_1^{\ell_1}\cdots q_n^{\ell_n}a_{\ell_1,\ldots,\ell_n}$,
	since $\{1, K, L, KL\}$ is an orthonormal basis of $\mathbb H$ over $\mathbb R$,
 we can write 
	\[a_{\ell_1,\ldots,\ell_n}=\alpha_{\ell_1,\ldots,\ell_n}+\beta_{\ell_1,\ldots,\ell_n}L\]
	with $\alpha_{\ell_1,\ldots,\ell_n},\beta_{\ell_1,\ldots,\ell_n} \in \mathbb C_K$, so that the splitting components of $P$ are
	\[F(z_1,\ldots,z_n)=\sum_{\ell_1,\ldots,\ell_n}z_1^{\ell_1}\cdots z_n^{\ell_n}\alpha_{\ell_1,\ldots,\ell_n}, \ \ \text{ and }  \ \
	G(z_1,\ldots, z_n)=\sum_{\ell_1,\ldots,\ell_n}z_1^{\ell_1}\cdots z_n^{\ell_n}\beta_{\ell_1,\ldots,\ell_n}.\] 
	Thanks to the fact that for any quaternions $u,v$ we have $\overline{uv}=\bar v \bar u$ and to the fact that orthogonal imaginary units anticommute, we have
	\[\overline{a_{\ell_1,\ldots,\ell_n}}=
	\overline{\alpha_{\ell_1,\ldots,\ell_n}} + \overline{\beta_{\ell_1,\ldots,\ell_n}L}= \overline{\alpha_{\ell_1,\ldots,\ell_n}} -L\overline{\beta_{\ell_1,\ldots,\ell_n}}= \overline{\alpha_{\ell_1,\ldots,\ell_n}} -\beta_{\ell_1,\ldots,\ell_n} L.\] 	
	Hence
	\[P^c(z_1,\ldots,z_n)= 
	\sum_{\ell_1,\ldots,\ell_n}z_1^{\ell_1}\cdots z_n^{\ell_n}(\overline{\alpha_{\ell_1,\ldots,\ell_n}} -\beta_{\ell_1,\ldots,\ell_n} L)=\overline{F(\bar z_1,\ldots,\bar z_n)}-G(z_1,\ldots,z_n)L.\]

\end{proof}

}
﻿
\noindent {In what follows, we will mainly consider {\em right ideals} of $\mathbb{H}[q_1,\ldots,q_n]$. In particular, we will use the notation 
	$\langle P : \ P \in \mathcal{E}\rangle $
	for the {\em right} ideal generated in $(\mathbb H[q_1,\ldots, q_n], +, *)$ by the slice regular polynomials belonging to the subset $ \mathcal E \subset \mathbb H[q_1,\ldots, q_n]$.}
	
\begin{definition}
	
	\noindent  Given a right ideal $I$ in $\mathbb{H}[q_1,\ldots,q_n]$,
	we define $\mathcal{V}(I)$ to be the set of common zeros of $P\in I$, i.e.,
	if $Z_P\subset \mathbb{H}^n$ denotes the zero set of a slice regular  polynomial $P\in I$, then 
	\[ \mathcal{V}(I):=\bigcap_{P\in I} Z_P.\]
	
	\noindent Furthermore, we set
	
	\[ \mathcal{V}_c(I):=\mathcal{V}(I)\cap\bigcup_{K\in \mathbb{S}}\mathbb{C}_K^n.\]

\end{definition}
\noindent Notice that for any right ideal $I$ the set $\mathcal V_c(I)$  is contained in $\mathcal{V}(I)$. 
﻿
﻿
\begin{definition}\label{idealinuovi}
	Let $Z$ be a non-empty subset of $\mathbb{H}^n$.
	\noindent We denote by
	$\mathcal{J}(Z)$  the right ideal generated in $\mathbb{H}[q_1,\ldots,q_n]$
	by slice regular polynomials which vanish on $Z$,
	\[\mathcal{J}(Z):=\left\{\sum_{k=1}^N P_k*Q_k\ :\ P_k,Q_k\in\mathbb{H}[q_1,\ldots,q_n]\ {\rm with}\ {P_k}_{|_Z}\equiv 0\right\}.\] 
\end{definition}
\noindent In general 
$\mathcal{J}(Z)$ does not coincide with the set of polynomials vanishing on $Z$, but if $Z=\mathcal{ V }_c(I)$ where $I$ is a right ideal of $\mathbb{H}[q_1,\ldots,q_n]$, then it can be proved directly, using Proposition \ref{prodstar}, that
\begin{equation}\label{jvci}
	\mathcal J(\mathcal V_c(I))=\{P\in \mathbb H[q_1,\ldots,q_n] : \ P_{|_{\mathcal V_c(I)}} = 0 \}.
\end{equation}
\noindent  We can say also that, if $Z_c$ is any subset of $\mathbb H^n$ of points with commuting components $$
	\mathcal J(Z_c)=\{P\in \mathbb H[q_1,\ldots,q_n] : \  P_{|_{Z_c}} = 0 \}.$$
\noindent Furthermore, notice that if $Z'\subseteq Z$ then $\mathcal J(Z')\supseteq \mathcal J (Z)$.
﻿
\noindent Combining Theorem 4.1 and Corollary 4.2 in \cite{Nul2} we have 
\begin{proposition} \label{vanishing}
	Let $I$ be a right ideal of $\mathbb H [q_1,\ldots,q_n]$. Then
\begin{enumerate}
	\item $\mathcal{J}(\mathcal{V}_c(I))=\mathcal{J}(\mathcal{V}(I))$; 
	\item $\mathcal{J}(\mathcal{V}(I))$ coincides with the set of polynomials vanishing on $\mathcal V(I)$.
\end{enumerate}
\end{proposition}
\noindent Therefore it is not difficult to show that, if $I_1$ and $I_2$ are two ideals in $\mathbb H[q_1,\ldots,q_n]$,
\begin{equation}\label{Jcap1}
	\mathcal{J}(\mathcal{V}(I_1)\cup \mathcal{V}(I_2))=\mathcal{J}(\mathcal{V}(I_1))\cap \mathcal{J}(\mathcal{V}(I_2)),
\end{equation}
and
\begin{equation}\label{Jcap}
	\mathcal{J}(\mathcal{V}_c(I_1)\cup \mathcal{V}_c(I_2))=\mathcal{J}(\mathcal{V}_c(I_1))\cap \mathcal{J}(\mathcal{V}_c(I_2)).
\end{equation}
\noindent To state the Strong Nullstellensatz for slice regular polynomials we need the following generalized notion of {\em prime} ideal in the non-commutative setting (see \cite{reyes}).
\begin{definition}\label{cprime} A right ideal $I$  in $\mathbb{H}[q_1,\ldots,q_n]$  is said to be  {\em completely prime} if and only if given two polynomials $P,Q$ such that 
	$P*Q\in I$ 
	and  $P*I\subseteq I$, 
	then  either $P\in I$ or $Q\in I$.
\end{definition}
﻿
﻿
﻿
\begin{theorem}[{Strong Nullstellensatz}]\label{snull}
	Let $I$ be a right ideal of $\mathbb{H}[q_1,\ldots,q_n]$. Then
	\[\mathcal{J}(\mathcal{V}(I))=\sqrt I:= \bigcap_{ I \subset J\\ \; {J\mathrm{ completely\ prime}}}J\]
\end{theorem}
﻿
﻿
﻿
﻿
﻿
﻿
﻿
\begin{definition} An ideal $I$  in $\mathbb{H}[q_1,\ldots,q_n]$  is said to be  {\em radical} if and only if $\sqrt I=I$.
\end{definition}
﻿
\noindent Notice that $ I$ is contained in $\sqrt{I}$.
﻿
﻿
\begin{remark} By definition, any completely prime ideal $I$  in $\mathbb{H}[q_1,\ldots,q_n]$ is a radical ideal. The converse is not true: the ideal $I$ in $\mathbb{H}[q]$ generated by the polynomial $q^2+1$ is radical. Indeed $$\sqrt{I}=\mathcal{J}(\mathcal{V}(I))=\mathcal{J}(\mathbb{S}_i)=I,$$
	but it is not completely prime, thanks to \cite[Lemma 4.7]{israeliani-1}.
 \end{remark}
﻿
%
﻿
\begin{proposition}\label{vrad} We have
	\[{\mathcal{V}_c(I)=\mathcal{V}_c(\sqrt{I}) }\] and
	\[{\mathcal{V}(I)=\mathcal{V}(\sqrt{I}) }\]
\end{proposition}	
\begin{proof}
	Since $\sqrt{I}\supseteq I$, then $\mathcal V_c(\sqrt{I})\subseteq \mathcal V_c({I})$ and $\mathcal V(\sqrt{I})\subseteq \mathcal V({I})$.
	The other inclusion follows from the fact that $\mathcal J(\mathcal V(I))$ coincides with the set of polynomials vanishing on $\mathcal V(I)$, and the same holds for $\mathcal V_c(I)$ (see Proposition \ref{vanishing}).
\end{proof}
﻿
\noindent Besides its algebraic meaning, the Strong Nullstellesantz is relevant also from a geometric viewpoint. In fact it allows to introduce a notion of algebraic set in this quaternionic setting which gives rise to a Zariski type topology, see \cite{Nul2}.
﻿
	\begin{definition} \label{slicealg}
		A subset $V\subseteq \mathbb H^n$ is called {\em slice algebraic} if for any $K\in \mathbb S$, $V\cap \mathbb C_K^n$ is a complex algebraic subset of $\mathbb C_K^n$.
	\end{definition}
﻿
\noindent In particular, one can prove (see \cite[Theorem 4.6]{Nul2}) that if $I$ is a right ideal of $\mathbb H [q_1,\ldots,q_n]$, then
$\mathcal V(I)$ is a slice algebraic set in $ \mathbb H^n$.	
Moreover, since for any $K\in \mathbb S$
\[\mathcal V_c(I) \cap \mathbb C_K^n=\mathcal V(I) \cap \mathbb C_K^n,\]
Definition \ref{slicealg} implies that {any set of the form} $\mathcal V_c(I)$ is slice algebraic as well.
﻿
﻿
\section{Symmetrization of algebraic sets and of ideals}
﻿
\noindent To begin with, we give the following 
﻿
{ \begin{definition} 
	Starting from a set $Z_c\subseteq \mathbb H^n$ containing only points with commuting components, we denote its {\em symmetrization} by $$\mathbb{S}_{Z_c}:=\bigcup_{\textbf{a} \in Z_c}\mathbb S_{\textbf a}.$$\end{definition}
}
﻿
\noindent {Note that, if all the points of $Z_c$ have commuting components, then the same holds true for all the points of $\mathbb{S}_{Z_c}$.﻿}
\noindent Let us recall the following propositions which enlighten the analogies between spheres in $\mathbb H$ and spherical sets in $\mathbb H^n$ as vanishing sets for slice regular polynomials in one and several quaternionic variables.
﻿
\begin{proposition}[{{\cite[Proposition 3.12]{Nul2}}}] \label{zerisferici}
	If a slice regular polynomial with real coefficients $R\in
	\mathbb R[q_1,\ldots,q_n]$ vanishes at a point
	$\textbf{a}\in \mathbb H^n$, then it vanishes on the
	entire spherical set $\mathbb S_{\textbf{a}}$.
\end{proposition}
\begin{proposition}[{\cite[Proposition 3.18]{Nul2}}]\label{zericinesi}
	Let $P\in \mathbb{H}[q_1,\ldots,q_n]$ and let ${\textbf{a}}\in \mathbb{C}_J^n$ for some $J\in \mathbb{S}$.
	If $P$ vanishes on two different points of the arranged spherical set $\mathbb S_{\textbf{a}}$, then $P$ vanishes on the entire $\mathbb S_{\textbf{a}}$.
\end{proposition}	
﻿
\noindent Thanks to the previous propositions, we can prove the following result, which generalizes to the several variable case Proposition 3.10 in \cite{libroGSS}. 
﻿
﻿
\begin{proposition}\label{corrsimm}
	Let $\textbf{a}\in \mathbb C^n_J$. Then the zeros of $P$ on $\mathbb{S}_{\textbf{a}}$ are in one-to-one correspondence with those of $P^c$. \\
	Moreover $P$  has a zero on $\mathbb{S}_{\textbf{a}}$ if and only if $P^s$ has a zero on $\mathbb{S}_{\textbf{a}}$ (and hence vanishes identically on it).
	
\end{proposition}
\begin{proof}
	Let $\textbf{\em b}\in \mathbb S_{\textbf{\em a}}$ be such that $P(\textbf{\em b})=0$. Then $P^s=P*P^c$ vanishes on $\mathbb S_{\textbf{\em a}}$, and in particular $P^s(\overline{ \textbf{\em b}})=0$. 
	Suppose first that $P(\overline{\textbf{\em b}})\neq 0$, { so that the only zero of $P$ on $\mathbb S_{\textbf{\em a}}$ is $\textbf{\em b}$.} Then
	\[0=P^s(\overline{ \textbf{\em b}})=P(\overline{ \textbf{\em b}})\cdot P^c(P(\overline{ \textbf{\em b}})^{-1}\overline{\textbf{\em b}}P(\overline{ \textbf{\em b}})),\]
	and thus $P^c$ vanishes on $P(\overline{ \textbf{\em b}})^{-1}\overline{\textbf{\em b}}P(\overline{ \textbf{\em b}}) \in \mathbb S_{\textbf{\em a}}$.
	If instead $P(\overline{\textbf{\em b}})= 0$, then, $P$ vanishes on the entire $\mathbb S_{\textbf{\em a}}$. 
	Let $K,L$ be two orthogonal imaginary units and consider the splitting $P=F+GL$. Since $P$ vanishes on $\mathbb S_{\textbf{\em a}}\cap \mathbb C_{K}^n=\{\textbf{\em c},\overline{\textbf {\em c}}\}$, we get that $F(\textbf{\em c})=F(\overline{\textbf{\em c}})=0$ and $G(\textbf{\em c})=G(\overline{\textbf{\em c}})=0$. 
	Using the splitting formula for $P^c$, we get that $P^c$ vanishes both at ${\textbf{\em c}}$ and at $\overline{\textbf{\em c}}$, and hence on the entire $\mathbb S_{\textbf{\em a}}$.  { Since the regular conjugation is an involution in $\mathbb H[q_1,\ldots,q_n]$ (see Definition \ref{R-coniugata2}), we get the claim. }
\end{proof}
﻿
﻿
\noindent Let us now investigate geometric and algebraic properties of the symmetrization of a slice algebraic set of the form $\mathcal V_c(I)$. \\
\noindent We need the following definition.
\begin{definition} 
	Given a right ideal $I\subseteq  \mathbb H[q_1,\ldots,q_n]$, we can define the {\em symmetrized ideal} $\mathcal{S}(I)$ as the ideal generated by $P^s$ with $ P\in I$, i.e. 
	$\mathcal{S}(I)=\langle P^s\;|\;P\in I\rangle$.
\end{definition}
\noindent Notice that $ \mathcal{S}(I)$ is contained in $I$. Moreover, since $\mathbb H[q_1,\ldots,q_n]$ is a Noetherian ring (see, e.g., \cite{reyes}), the symmetrized ideal $\mathcal S(I)$ can be generated by a finite number of symmetrized polynomials. 
﻿
\begin{theorem}\label{simmalg}
	Let $I$ be a right ideal of $\mathbb H[q_1,\ldots,q_n]$.  Then the symmetrization of $\mathcal V_c(I)$ is a slice algebraic set, more precisely
	\begin{equation}\label{inclusion}
		\mathbb{S}_{\mathcal{V}_c(I)}=	\mathcal{V}_c( \mathcal{S}(I)).
	\end{equation}
	
\end{theorem}
﻿
\begin{proof} The inclusion $\mathbb{S}_{\mathcal{V}_c(I)}\subseteq\mathcal{V}_c( \mathcal{S}(I))$ can be proved easily. Indeed
	every polynomial $P$ in $I$ vanishes on $\mathcal{V}_c(I)$, and thus $P^s$ vanishes on $\mathbb{S}_{\mathcal{V}_c(I)}$. \\
	For the other inclusion, assume by contradiction that there exists  $\textbf{\em a}\in \mathcal{V}_c(\mathcal{S}(I))\setminus \mathbb{S}_{\mathcal{V}_c(I)}$. 
	Since $\textbf{\em a}\notin \mathbb{S}_{\mathcal{V}_c(I)}$, we get that $\textbf{\em a}\notin \mathbb{R}^n$ and $\mathbb S_{\textbf{\em a}} \cap \mathcal{V}_c(I)=\varnothing$. Hence there exist at least two polynomials  $P_1$ and $P_2$ in $I$ without a common zero on $\mathbb{S}_{\textbf{\em a}}$. Call $\textbf{\em a}_1$ and $\textbf{\em a}_2$
	the two points on $\mathbb{S}_{\textbf{\em a}}$ on which $P_1$ and $P_2$ vanish respectively.  Note that $\textbf{\em a}_1\neq \textbf{\em a}_2.$ Observe that for all $\alpha,\beta\in \mathbb{H},$ $P_1*\alpha+P_2*\beta\in I$. Thus $F_{\alpha,\beta}=(P_1*\alpha+P_2*\beta)^s\in S(I)$ vanishes on 
	$\mathbb{S}_{\textbf{\em a}}$. We evaluate $F_{\alpha,\beta}=P_1^s|\alpha|^2+P_2^s|\beta|^2+P_1*\alpha*\overline{\beta}*P_2^c+P_2*\beta*\overline{\alpha}*P_1^c$ on $\textbf{\em a}_1$.
	{If $\alpha,\beta\in \mathbb{H}\setminus\{0\}$}, using the notation of Proposition \ref{prodstar},	
	we get $$0=F_{\alpha,\beta}(\textbf{\em a}_1)=[P_2{ { *}}\beta*\overline{\alpha}*P_1^c] (\textbf{\em a}_1)=P_2(\textbf{\em a}_1)\beta\overline{\alpha}\cdot P_1^c({(P_2(\textbf{\em a}_1)\beta\overline{\alpha})}^{-1}\textbf{\em a}_1P_2(\textbf{\em a}_1)\beta\overline{\alpha}),$$ 
	{since $P_2(\textbf{\em a}_1)\beta\overline{\alpha}\neq 0$.}	 
	Denote by $\textbf{\em b}_1$ 	 
	the unique zero on $\mathbb{S}_{\textbf{\em a}}$ of $ P_1^c$. Thus we have {$\forall \alpha,\beta\in \mathbb{H}\setminus\{0\}$}
	
	$$\textbf{\em b}_1={(P_2(\textbf{\em a}_1)\beta\overline{\alpha})}^{-1}\textbf{\em a}_1P_2(\textbf{\em a}_1)\beta\overline{\alpha}=\overline{\alpha}^{-1}\beta^{-1}P_2(\textbf{\em a}_1)^{-1}\textbf{\em a}_1P_2(\textbf{\em a}_1)\beta\overline{\alpha}.$$
	For $\alpha=\beta=1$ we get 
	$$\textbf{\em b}_1={(P_2(\textbf{\em a}_1))}^{-1}\textbf{\em a}_1P_2(\textbf{\em a}_1),$$ 
	and hence for $\alpha=1, \beta\notin \mathbb{R}$ 
	$${\beta}^{-1} \textbf{\em b}_1 \beta=\textbf{\em b}_1.$$ Thus $\textbf{\em b}_1 \beta= \beta\textbf{\em b}_1$ but this can happen only if $\textbf{\em b}_1\in \mathbb{R}^n$. A contradiction since $\textbf{\em a}\notin\mathbb{R}^n$.
	
\end{proof}
﻿
\noindent From the algebraic viewpoint we have the following description.
\begin{corollary}\label{essealg}
	Let  $I$ be an ideal  in $\mathbb H[q_1,\ldots,q_n]$. Then 
	\[\mathcal J (\mathbb S_{\mathcal V_c(I)})=\sqrt{\mathcal S(I)}.\]
Furthermore $\mathcal J (\mathbb S_{\mathcal V_c(I)})$ is a two-sided ideal.
	\end{corollary} 
\begin{proof}
The equality $\mathcal J (\mathbb S_{\mathcal V_c(I)})=\sqrt{\mathcal S(I)}$ follows directly from Theorem \ref{simmalg} and from the Strong Nullstellesatz \ref{snull}.\\
For the second part, let $Q$ be any polynomial in $\mathbb{H}[q_1,\ldots,q_n]$; we have to show that $Q * \mathcal J (\mathbb S_{\mathcal V_c(I)}) \subseteq \mathcal J (\mathbb S_{\mathcal V_c(I)})$. Let $P \in \mathcal J (\mathbb S_{\mathcal V_c(I)})$, and let $\textbf{\em a} \in \mathbb S_{\mathcal V_c(I)}$. Then either $Q(\textbf{\em a})=0$, and hence $Q*P(\textbf{\em a})=0$; or
\[Q*P(\textbf{\em a})=Q(\textbf{\em a})P(Q(\textbf{\em a})^{-1}\textbf{\em a}Q(\textbf{\em a}))=0\]
since $P$ vanishes on $\mathbb S_{\mathcal V_c(I)}$. 
Therefore $Q*P \in \mathcal J (\mathbb S_{\mathcal V_c(I)})$.
 \end{proof}
﻿
\noindent The second part of the statement of the previous corollary can be easily generalized as follows.
\begin{proposition}\label{Zc}
	Let $Z_c$ be a subset of $\mathbb H^n$ of points with commuting components. Then $\mathcal J(\mathbb S_{Z_c})$ is a two-sided ideal.
\end{proposition}
﻿
﻿
﻿
﻿
﻿
﻿
%
﻿
﻿
﻿
﻿
\begin{definition}
	We say that a slice algebraic set $\mathcal V_c(I)$ has {\em spherical symmetry} if $\mathcal V_c(I)= { \mathcal V_c(\mathcal S(I))=}\mathbb S_{\mathcal V_c(I)}$.
\end{definition}
﻿
%
﻿
﻿
﻿
%
﻿
\noindent As a particular case, we have the following result.
\begin{proposition}\label{twosided}
	Let $I$ be a two-sided ideal of $\mathbb H[q_1,\ldots, q_n]$. Then $\mathcal V_c(I)$ has spherical symmetry.
\end{proposition}
\begin{proof}
	Let $\textbf{\em a} \in \mathcal V_c(I)$. We want to show that $\mathbb S_{\aaa}\subset  \mathcal V_c(I)$. Let $P\in I$, for any $b \in \mathbb H\setminus \{0\}$ we have that $b*P \in I$, hence, using the notation of Proposition \ref{prodstar},
	\[bP(b^{-1}{\aaa}b)=0,\]
	which implies that $P$ vanishes on the entire sphere $\mathbb S_\aaa$. 
\end{proof}
\noindent As a consequence we have
\begin{corollary} If an ideal $I\subseteq  \mathbb H[q_1,\ldots,q_n]$ is generated by polynomials with real coefficients, then $\mathcal V_c(I)$ has spherical symmetry. 
\end{corollary}
\noindent Furthermore we have the following equivalence.
\begin{proposition}\label{equivalence}
{The slice algebraic set}	$\mathcal V_c(I)$ has spherical symmetry if and only if $\sqrt{I}$ is a two-sided ideal.
\end{proposition}
\begin{proof}
Indeed $\sqrt{I}=\mathcal J(\mathcal V_c(I))=\mathcal J(\mathbb S_{\mathcal V_c(I)})$ which is two-sided, thanks to  Proposition \ref{essealg}.	
The other direction follows directly from Proposition \ref{twosided}, recalling that $\mathcal V_c(I)=\mathcal V_c(\sqrt{I})$.
﻿
﻿
	\end{proof}
\section{On the correspondence between algebraic sets and ideals}
\noindent Given two { right} ideals $I_1,I_2$ of $\mathbb{H}[q_1,\ldots,q_n]$, we define the $*$-product ideal $I_1*I_2$ to be the { right} ideal generated by the $*$-product of polynomials of 
$I_1$ and $I_2$, namely $I_1*I_2:=\langle P*Q,\ P\in I_1, \ Q\in I_2\rangle$.
﻿
\noindent We have the following relations concerning algebraic sets associated with $*$-product, intersection and sum of two ideals.   
\begin{proposition}\label{corr}
	Let $I_1,I_2$ be { right} ideals of $\mathbb H[q_1,\ldots,q_n]$. Then
	\begin{enumerate}
		\item  $\mathcal{V}_c(I_1*I_2)\subseteq \left(\mathcal{V}_c(I_1)\cup\mathcal{V}_c( \mathcal{S}(I_2))\right)$;
		\item $\mathcal V_c(I_1 \cap I_2)\supseteq \left(\mathcal{V}_c(I_1)\cup\mathcal{V}_c(I_2) \right)$;
		\item $\mathcal{V}_c(I_1+I_2)=\left(\mathcal{V}_c(I_1)\cap\mathcal{V}_c(I_2)\right)$.
	\end{enumerate}
\end{proposition}
\begin{proof}
The proof follows taking into account the expression of the $*$-product of polynomials evaluated on points with commuting components given in Proposition \ref{prodstar} and the fact that the operator $\mathcal V$ reverses the inclusions.	
	\end{proof}
\noindent In the particular case in which one of the two ideals is generated by polynomials with real coefficients we have the following result.
\begin{proposition}
	Let $I_1,I_2$ be { right} ideals of $\mathbb H[q_1,\ldots,q_n]$ such that $I_2=\langle G_1,\ldots,G_n\rangle$ with $G_\ell \in \mathbb R[q_1,\ldots,q_n]$ for any $\ell=1,\ldots,n$. Then
	\[\mathcal V_c(I_1 \cap I_2)=\mathcal V_c(I_1 * I_2)=\mathcal V_c(I_1)\cup \mathcal V_c(I_2).\]
\end{proposition}
\begin{proof}
	The fact that $I_2$ is generated by polynomials with real coefficients, guarantees that $I_1*I_2\subseteq I_1\cap I_2 $, and hence 
	\[\mathcal V_c(I_1\cap I_2)\subseteq \mathcal V_c(I_1*I_2). \]
	Consider now $\textbf{\em a}\in \mathcal V_c(I_1 * I_2)$. Then, for any $F\in I_1, G \in I_2$, $F*G(\textbf{\em a})=0$. If $\textbf{\em a}\not \in \mathcal V_c(I_1)$, then there exists $F_0\in I_1$ such that $F_0(\textbf{\em a})\neq 0$. Then, recalling Proposition \ref{prodstar}, for any $\ell=1,\ldots,n$,  we have \[0=F_0(\textbf{\em a})G_\ell(F_0(\textbf{\em a})^{-1}\textbf{\em a}F_0(\textbf{\em a}))=G_\ell(\textbf{\em a})F_0(\textbf{\em a})\]
	where the last equality follows from the fact that $G_\ell$ has real coefficients. Hence $G_\ell(\textbf{\em a})=0$ for any $\ell$, that is $\textbf{\em a}\in \mathcal V_c(I_2)$. 
	Thus \[\mathcal V_c(I_1 * I_2)\subseteq \mathcal V_c(I_1)\cup \mathcal V_c(I_2).\]
	To complete the proof, notice that, since $I_1 \cap I_2$ is contained in $I_1$ and in $I_2$, we have 
	\[\mathcal V_c(I_1)\cup \mathcal V_c( I_2)\subseteq \mathcal V_c(I_1\cap I_2).\] 		 
\end{proof}	
\noindent As a particular case, we have
\begin{corollary}\label{essei}
		Let $I_1,I_2$ be {right} ideals of $\mathbb H[q_1,\ldots,q_n]$. Then
		\[\mathcal{V}_c(I_1* \mathcal{S}(I_2))=\mathcal{V}_c(I_1)\cup\mathcal{V}_c( \mathcal{S}(I_2))=\mathcal{V}_c(I_1\cap \mathcal{S}(I_2)).\]
\end{corollary}
%
﻿
﻿
%
\section{Irreducible slice algebraic sets}\label{S1}
﻿
﻿
﻿
﻿
﻿
﻿
﻿
﻿
﻿
﻿
﻿
﻿
﻿
﻿
﻿
﻿
﻿
\noindent ﻿{Let us begin by giving the definition of reducibility of slice algebraic sets.}
﻿
\begin{definition} \label{red}

Given a right ideal $I\subseteq \mathbb{H}[q_1,\ldots,q_n]$, the slice algebraic set $\mathcal{V}(I)$, 
is said to be {\em reducible} whenever there exist two proper, non-trivial, right ideals $I_1$ and $I_2$ in $\mathbb{H}[q_1,\ldots,q_n]$, not contained one into another, such that 
$$\mathcal{V}(I)=\mathcal{V}(I_1)\cup \mathcal{V}(I_2).$$
Analogously, $\mathcal{V}_c(I)$ is said to be {\em reducible} if there exist two proper, non-trivial, right ideals $I_1$ and $I_2$ not contained one into another, such that 
$$\mathcal{V}_c(I)=\mathcal{V}_c(I_1)\cup \mathcal{V}_c(I_2).$$
﻿
\noindent If $\mathcal{V}(I)$ (or $\mathcal{V}_c(I)$) is not reducible, it is called {\em irreducible}.
﻿
\end{definition} %
\noindent Thanks to Proposition 3.14 in \cite{Nul2}, we have the following result. 

{\begin{proposition}\label{GeneratoriSfere}
	Let $\mathbb{S}_{(a_1,\cdots,a_n)}\subset\mathbb{H}^n$ be an arranged spherical set and suppose $a_\ell \not \in \mathbb{R}$ for any $\ell$. Then $\mathbb S_{(a_1,\ldots, a_n)}=\mathcal{V}(I)$ where $I$ is a two-sided radical ideal generated by polynomials with real coefficients.
\end{proposition}
}
\noindent Moreover 
﻿
\begin{proposition}\label{S_a} Any arranged spherical set $\mathbb{S}_{(a_1,\cdots,a_n)}\in\mathbb{H}^n$ { with $a_\ell \not \in \mathbb{R}$ for any $\ell$} is an irreducible slice algebraic set.
\end{proposition}
\begin{proof}
{ Proposition \ref{GeneratoriSfere} yields that arranged spherical sets are always slice algebraic sets of the form $\mathcal V(I)$.} Suppose now, by contradiction, that $\mathbb{S}_{(a_1,\cdots,a_n)}=\mathcal{V}(I_1)\cup\mathcal{V}(I_2)$ is reducible. Then, without loss of generality, there exist at least two points $(a'_1,\ldots, a_n')$ and $(a''_1,\ldots, a_n'')\in \mathcal{V}(I_1)\subset \mathcal{V}(I)$. Since they belong to the same spherical set, the entire sphere through $(a'_1,\ldots, a_n')$ belongs to $\mathcal{V}(I_1)$, that is $\mathcal{V}(I)=\mathbb{S}_{(a'_1,\ldots, a_n')}=\mathbb{S}_{(a''_1,\ldots, a_n'')}\subset \mathcal{V}(I_1)$. 
	\end{proof}
\noindent {Another family of irredudible slice algebraic set is given by the following proposition.}
\begin{proposition}\label{SxS} Any set of the form  $\mathbb{S}_{a}\times\mathbb{S}_{b}\subset \mathbb H^2$ with $a,b\notin \mathbb R$ 
	is an irreducible algebraic set. 
	\end{proposition}
\begin{proof}
Combining Propositions 3.4 and 3.7 in \cite{Nul2}, we have that $\mathbb{S}_{a}\times\mathbb{S}_{b}=\mathcal V(I)$ where $I$ is the (radical) ideal generated by $(q_1-a)^s,(q_2-b)^s$.
	Suppose {by contradiction} that $\mathcal V(I)=\mathcal V(I_1)\cup \mathcal V(I_2)$ is reducible. Thanks to Proposition \ref{vrad} we can assume that $I_1$ and $I_2$ are radical ideals. For the points with commuting components we {then} have
	\[\mathcal V_c(I)=\mathcal V_c(I_1)\cup \mathcal V_c(I_2)=\mathbb S_{(a,\tilde{b})}\cup \mathbb S_{(\bar{a},\tilde{b})}\]
	where $a\tilde{b}=\tilde{b}a$.  
	Since 	$\mathbb S_{(a,\tilde{b})}$ and $\mathbb S_{(\bar{a},\tilde{b})}$ are irreducible, up to swapping the subsets, we get that
	$\mathcal V_c(I_1)=\mathbb S_{(a,\tilde{b})}$ and $\mathcal V_c(I_2)=\mathbb S_{(\bar{a},\tilde{b})}$. Thus, recalling Proposition \ref{GeneratoriSfere}, $I_1$ and $I_2$ are such that $\mathcal V(I_1)=\mathcal V_c(I_1)$ and $\mathcal V(I_2)=\mathcal V_c(I_2)$. 
Therefore
	\[\mathbb S_a \times \mathbb S_b=\mathcal V(I_1)\cup \mathcal V(I_2)=\mathcal V_c(I_1)\cup \mathcal V_c(I_2)=\mathbb S_{({a},\tilde{b})}\cup \mathbb S_{(\bar{a},\tilde{b})},\]
	which is a contradiction.
\end{proof}
﻿
\noindent The previous proposition easily generalizes to any subset of the form $\mathbb{S}_{(a_1, \ldots,a_k)}\times \mathbb S_{(b_{k+1},\ldots,b_n)}$ in $\mathbb H^n$.\\
﻿
\noindent Let us now give a couple of examples of reducible slice algebraic sets. 
\begin{example}\label{red1}
	The subset of $\mathbb{H}^2$ given by  $V=\{(i,i)\}\cup \{(2j,2j)\}$ is a reducible slice algebraic set. 
	\noindent
	Indeed $V=\mathcal{V}(I)$ where $I=\langle q_1-q_2, (q_1-i)*(q_2-2\tilde{J})\rangle$.
	\noindent This can be seen  by considering a polynomial $P$ that vanishes on $V$; if we divide $P$  by $-q_2+q_1$ (as a monic polynomial in the variable $q_2$), we obtain a rest $R(q_1)$ which vanishes again on $V$, that is for $q_1=i$ and $q_1=2j$. 
	Recalling Proposition 3.21 in \cite{libroGSS}, 
	 $$R(q_1)=(q_1-i)*(q_1-2\tilde{J})*R_1(q_1)$$ where $\tilde{J}=(2j-i)^{-1}j(2j-i)=\frac{-4i+3j}{5}$.   
	 Thus $P$ can be written as a combination  of $q_1-q_2$ and  $(q_1-i)*(q_2-2\tilde{
		J})$. \\
		Therefore we have that $V=\mathcal{V}(I)=\mathcal{V}(I_1) \ \cup \ \mathcal{V}(I_2)$ is reducible,  with $I_1=\langle (q_1-q_2), (q_1-i)\rangle$ and $I_2=\langle (q_1-q_2), (q_2-2j)\rangle$.
\end{example}
﻿
\begin{example} \label{red2}  
	Let $I=\langle q_1^2+1,(q_1-q_2)*(q_2-i)\rangle$. Then 
	 $$\mathcal{V}(I)=(\mathbb{S}\times\mathbb{H})\cap  (\{q_1=q_2\}\cup (\mathbb{H}\times \{i\}))$$
	is reducible since $\mathcal{V}(I)=\mathcal{V}(I_1)\cup \mathcal{V}(I_2)$ where $I_1=\langle q_1^2+1,q_1-q_2\rangle$ and $I_2=\langle q_1^2+1,q_2-i\rangle $. 
	Indeed
	$\mathcal{V}(I_1) =\mathbb{S}_{(i.i)}=(\mathbb{S}\times\mathbb{H})\cap  \{q_1=q_2\}$ and $\mathcal{V}(I_2)=\mathbb{S}_i\times \{i\}=(\mathbb{S}\times\mathbb{H})\cap   (\mathbb{H}\times \{i\})$. 
\end{example}
﻿
\noindent {﻿Let us now describe the relation between the reducibility of a slice algebraic set and the reducibility of its points with commuting components. 
}\begin{proposition}\label{vc-v}
{Let $I$ be a right ideal of $\mathbb H[q_1,\ldots,q_n]$.}
If $\mathcal{V}_c(I)$ is irreducible, then $\mathcal{V}(I)$ is irreducible.
\end{proposition}
\begin{proof}
	Suppose by contradiction that $\mathcal V(I)=\mathcal{V}(I_1)\cup \mathcal{V}(I_2)$ is reducible. Then 
	$\mathcal{V}_c(I)=\mathcal{V}_c(I_1)\cup \mathcal{V}_c(I_2)$, where $\mathcal V_c(I_1)$ and $\mathcal V_c(I_2)$ are not empty thanks to the Weak Nullstellensatz (see \cite[Theorem 1.1]{israeliani-1}). {Since $\mathcal{V}_c(I)$ is irreducible, we can suppose without loss of generality that $\mathcal{V}_c(I_1)\subseteq \mathcal{V}_c(I_2)$.} Recalling the Strong Nullstellensatz \ref{snull}, we then have
	
	\[\sqrt{I_1}=\mathcal J(\mathcal V(I_1))=\mathcal J(\mathcal V_c(I_1)) \supseteq \mathcal J(\mathcal V_c(I_2))=\mathcal J(\mathcal V(I_2))=\sqrt{I_2},\]
which implies that
\[\mathcal V(I_1)=\mathcal V(\sqrt{I_1})\subseteq \mathcal V(\sqrt{I_2})=\mathcal V({I_2}),\]
in contradiction with the reducibility of $\mathcal V(I)$.
\end{proof}
﻿{
\begin{remark}
	The converse of the previous proposition is not true in general. Consider for instance $\mathbb{S}_i\times \mathbb{S}_i$, which is an irreducible slice algebraic set of the form $\mathcal{V}(I)$ with $I=\langle(q_1-i)^s,(q_2-i)^s\rangle$, thanks to  Proposition \ref{SxS}. In this case 
	$\mathcal{V}_c(I)=\mathbb S_{(i,i)}\cup \mathbb S_{(-i,i)}$, which is a reducible slice algebraic set, thanks to Proposition \ref{S_a}.
\end{remark}}
%
﻿
﻿
%
%
﻿
﻿
﻿
\noindent Irreducible complex algebraic sets are characterized by the fact that their associated ideals are {\em prime}. From the quaternionic Strong Nullstellesatz \ref{snull}, one might think that {\em completely prime} ideals could play the analogue role in our setting. 
This is not the case, as shown by the ideal $I$ generated by $q^2+1$ in $\mathbb H[q]$: {as already recalled in the Introduction, $I$ is not completely prime, while its associated slice algebraic set is $\mathbb S_i$ which is irreducible.}
﻿
\noindent {Motivated by this gap in the correspondence}, we introduce a new class of ideals, related to the geometry of the zero sets of slice regular polynomials.    
﻿
\begin{definition}\label{qpideal} A right ideal $I\subseteq  \mathbb H[q_1,\ldots,q_n]$ is said to be  {\em quasi prime} if and only if given two polynomials $P,Q$ such that 
$P*Q\in I$, 
then  either $P\in I$ or $Q^s\in I$.
 \end{definition}
﻿
\noindent The class of quasi prime ideals contains the one of completely prime ideals.
 
\begin{proposition} \label{cp-qp}
If $I\subseteq  \mathbb H[q_1,\ldots,q_n]$ is a completely prime right ideal, then $I$ is a quasi prime ideal. 
\end{proposition}
\begin{proof} Let $P,Q$ be {slice regular polynomials }such that $P*Q\in I$. Then $P*Q*Q^c=P*Q^s=Q^s*P\in I$, moreover $Q^s*I=I*Q^s\subseteq I$. The fact that $I$ is completely prime implies that either $P\in I$ or $Q^s\in I$.
\end{proof}
﻿
﻿
﻿
﻿
\noindent 
From the irreducibility of a slice algebraic set we can deduce some algebraic properties of the corresponding ideal. 
\begin{theorem}\label{VirredIprime} Let $I$ be a radical ideal of $\mathbb{H}[q_1,\ldots,q_n]$ such that $\mathcal{V}_c(I)$ is irreducible. Then $I$ is a quasi prime ideal.
 \end{theorem} 
%
\begin{proof} 
	Suppose  $P*Q\in I$, with $Q$ a non-constant polynomial. We have that $P*Q$ vanishes on  $\mathcal{V}(I)$, and thus it vanishes also on $\mathcal{V}_c(I).$ On this set $P*Q$ has a manageable expression. Indeed for all $\textbf{\em a}=(a_1,a_2\cdots,a_n)\in \mathcal{V}_c(I)$, either $P(\textbf {\em a})=0$, or
		 $$(P*Q)(\textbf{\em a})=P(\textbf{\em a})\cdot Q({P}^{-1}(\textbf{\em a})\textbf {\em a}P(\textbf{\em a})),$$
	that is $Q({P}^{-1}(\textbf{\em a})\textbf{\em a}P(\textbf{\em
		a}))=0$. 
		Take now the two subsets of $\mathcal{V}_c(I),$ given
	by $$V_1=\mathcal{V}(I)\cap \mathcal{V}_c(\langle P\rangle)=\mathcal{V}_c(I+\langle P\rangle)\subseteq \mathcal{V}_c(\langle P\rangle)$$
	and
	$$  V_2=\mathcal{V}(I)\cap \mathbb{S}_{\mathcal{V}_c(\langle
		Q\rangle)}=
	\mathcal{V}_c{(I)}\cap\mathcal{V}_c(\langle
	{Q^s}\rangle)=\mathcal{V}_c(I+\langle Q^s\rangle)\subseteq
	\mathcal{V}_c(\langle {Q^s}\rangle).$$
	Now  either ${\textbf{\em a}}$ belongs to $V_1$ or one of
	its conjugates belongs to $V_c(\langle Q\rangle)$, thus ${\textbf{\em a}}$
	belongs to $V_2$; therefore $$\mathcal{V}_c(I)= V_1\cup {V_2}.$$ 
	Since $\mathcal{V}_c(I)$ is irreducible,{ the first possibility is that
	$$\mathcal{V}_c(I)=V_1\subseteq \mathcal{V}_c(\langle P\rangle).$$ 
	In this case $P$ vanishes on $\mathcal{V}_c(I),$ and hence $P\in \mathcal{J}(\mathcal{V}_c(I))=I$ since $I$ is radical; the second possibility is that $\mathcal{V}_c(I)=V_2\subseteq \mathcal{V}_c(\langle {Q^s}\rangle)$. In this case $Q^s\in  \mathcal{J}(\mathcal{V}_c(I))=I$, since $I$ is radical}.
\end{proof}
﻿
﻿
﻿
\noindent {The first result in the other direction is the following.}
﻿
 \begin{theorem} \label{nosfere}Let $I$ be a quasi prime and radical ideal of $\mathbb{H}[q_1,\ldots,q_n]$.
Then $\mathcal{V}_c(I)$ cannot be written as a non-trivial union of the form $\mathcal{V}_c(I_1)\cup \mathcal{V}_c(I_2)$ where  either $\mathcal{V}_c(I_1)$ or $\mathcal{V}_c(I_2)$ has spherical symmetry.
%
\end{theorem}
﻿
﻿
\begin{proof}
	 Suppose, by contradiction, that $\mathcal{V}_c(I)$ can be written as  $$\mathcal{V}_c(I)=\mathcal{V}_c(I_1)\cup \mathcal{V}_c(I_2)$$ where $\mathcal{V}_c(I_2)$ has spherical symmetry. {Suppose that there exists} $H\in \mathcal{J}(\mathcal{V}_c(I_1))\setminus{I}$  and take any $K\in  \mathcal{J}(\mathcal{V}_c(I_2))$.
	Consider the polynomial $H*K$, and take a point $\textbf{\em a}\in \mathcal{V}_c(I)$. {If $\textbf{\em a}\in \mathcal{V}_c(I_1)$, then  $H(\textbf{\em a})=0$ and hence $H*K(\textbf{\em a})=0$ as well; 
if,	otherwise, $\textbf{\em a}\in \mathcal{V}_c(I_2)$, then $\mathbb{S}_{\textbf{\em a}}\subseteq \mathcal{V}_c(I_2)$ since $\mathcal{V}_c(I_2)=\mathcal{V}_c( \mathcal{S}(I_2))$ and hence
$$H*K(\textbf{\em a})=H(\textbf{\em a})K(H^{-1}(\textbf{\em a})\textbf{\em a}H(\textbf{\em a}))=0.$$
	Thus $H*K\in \mathcal{J}(\mathcal{V}_c(I))=\sqrt{I}=I$. Recalling that $I$ is quasi prime and that $H\notin I$ we necessarily have that $K^s\in I$. 
	Since $K$ was chosen arbitrarily in $\mathcal{J}(\mathcal{V}_c(I_2))$, which contains $I_2$, we get that
	$$\mathcal{V}_c(I_2)\subseteq  \mathcal{V}_c(I) \subseteq \bigcap_{K\in I_2}{\mathcal{V}_c(\langle K^s\rangle)}= \mathcal{V}_c( \mathcal{S}(I_2))$$ 
	
\noindent Since by hypothesis  $\mathcal{V}_c(I_2)=\mathcal{V}_c( \mathcal{S}(I_2)),$ we have that $\mathcal{V}_c(I)=\mathcal{V}_c(I_2)$ and hence we conclude that $\mathcal{V}_c(I)$ is irreducible. }
	
\end{proof}
{\noindent As a consequence of the previous theorem we have the following.
{\begin{corollary}\label{coro}
Let $I$ be a quasi prime and radical right ideal of $\mathbb H[q_1,\ldots,q_n]$ such that $\mathcal{V}_c(I)$ has spherical symmetry. Then $\mathcal V_c(I)$ is irreducible.
\end{corollary}}
\begin{proof}
If $\mathcal V_c(I)=\mathbb S_{\mathcal V_c(I)}$ and $\mathcal V_c(I)=\mathcal V_c(I_1)\cup \mathcal V_c(I_2)$ is reducible, then $\mathbb S_{\mathcal V_c(I)}=\mathbb S_{\mathcal V_c(I_1)}\cup \mathbb S_{\mathcal V_c(I_2)}=\mathcal V_c(I_1)\cup \mathcal V_c(I_2)$. Therefore $\mathcal V_c(I_1)$ and $\mathcal V_c(I_2)$ consist of the union of arranged spheres. Then $\mathbb S_{\mathcal V_c(I_1)}=\mathcal V_c(I_1)$ and $\mathbb S_{\mathcal V_c(I_2)}=\mathcal V_c(I_2)$, in contradiction with Theorem \ref{nosfere}.
\end{proof}	
}
\noindent In view of Proposition \ref{equivalence}, Corollary \ref{coro} can be restated as follows
\begin{corollary}\label{coro2}
	If $I$ is a quasi prime and radical two-sided ideal of $\mathbb H[q_1,\ldots,q_n]$, then $\mathcal V_c(I)$ is irreducible.
\end{corollary}
﻿
\noindent 	Notice that if $\mathcal{V}_c(I)$ is irreducible and has spherical symmetry, then, either $\mathcal{V}_c(I)$ consists of a unique arranged sphere or it is the union of an infinite number of arranged spheres. {We will give an explicit example of this phenomenon at the end of Section \ref{Sprinc}.}  

\noindent 	Theorem \ref{nosfere} can be used as a tool to establish if an ideal is quasi prime, as shown by the following examples.
\begin{example}
	
The ideal $I=\langle q_1^2+1,q_2^2+1\rangle$ is not quasi prime. Indeed $\mathcal{V}(I)=\mathbb{S}\times \mathbb{S}$ is irreducible while  $\mathcal{V}_c(I)=\mathbb{S}_{(i,i)}\cup\mathbb{S}_{(-i,i)}$ is reducible.   Thanks to the Strong Nullstellensatz \ref{snull} we have that 
\[\sqrt{I}=\mathcal J(\mathbb{S}\times \mathbb{S})=\langle q_1^2+1,q_2^2+1\rangle=I,\]
that is $I$ is radical. Theorem \ref{nosfere} implies then that $I$ is not quasi prime.  
\end{example}
\begin{example}
	The ideal $I=\langle (q_1^2+1)*(q_2^2+1)\rangle$ is not quasi prime. In fact, as above, the Strong Nullstellensatz yields that $I$ is radical but $\mathcal{V}(I)=(\mathbb{S}\times\mathbb{H})\cup (\mathbb{H}\times\mathbb{S})$ and hence
	\[\mathcal{V}_c(I)=\bigcup_{\substack{x,y \in \mathbb R \\ I\in \mathbb S}}\mathbb{S}_{(I,x+yI)} \cup \bigcup_{\substack{x,y \in \mathbb R \\ I\in \mathbb S}}\mathbb{S}_{(x+yI,I)}.\]
	Thanks to Theorem \ref{nosfere}, the ideal $I$ is not quasi prime.
\end{example}
﻿
﻿
\noindent Another result in the direction of characterizing ideals whose corresponding algebraic set are irreducible is the following. 
{\begin{theorem}\label{cpirred}
		Let $I$ be a quasi prime and radical right ideal of $\mathbb H[q_1,\ldots,q_n]$ such that $\mathbb S_{\mathcal V_c(I)}$ is irreducible. Then $\mathcal V_c(I)$ is irreducible.
	\end{theorem}
	\begin{proof}
		Suppose by contradiction that $$\mathcal{V}_c(I)=\mathcal{V}_c(I_1)\cup \mathcal{V}_c(I_2)$$ is reducible. Then
		\[\mathbb S_{\mathcal{V}_c(I)}=\mathbb S_{\mathcal{V}_c(I_1)}\cup \mathbb S_{\mathcal{V}_c(I_2)}.\]
		Thus we can assume, without loss of generality, that  $\mathbb S_{\mathcal{V}_c(I_1)}\subseteq\mathbb S_{\mathcal{V}_c(I_2)}$. 
		This implies that for any ${\aaa} \in \mathcal{V}_c(I_1) \setminus \mathcal{V}_c(I_2)$ there exists ${\bf \tilde \aaa} \in \mathcal{V}_c(I_2) \cap \mathbb S_{\aaa}$, and hence the entire $\mathbb{S}_{\aaa}\subset \mathcal V_c(I)$. Since $\textbf{\em a} \notin \mathcal V_c(I_2)$, then necessarily $\mathbb S_{\aaa}$ is contained in $\mathcal{V}_c(I_1)$.
		Denote by $V_1=\mathbb S_{\mathcal V_c(I_1)\setminus \mathcal V_c(I_2)}$. Notice that $V_1 \subseteq \mathcal V_c(I_1)$ and  \[\mathcal V_c(I)=V_1\cup \mathcal V_c(I_2).\]
		Consider now a polynomial $P \in \mathcal J(\mathcal V_c(I_2))\setminus I$, and let $K$ be any polynomial in $\mathcal J(V_1)$. We want to prove that $P*K \in I$. For any $\aaa \in \mathcal V_c(I)$, if $\aaa \in \mathcal V_c(I_2)$, $P(\aaa)=0$ and hence $P*K(\aaa)=0$. Let $\aaa \in V_1\setminus \mathcal V_c(I_2)$. If $P(\aaa)=0$, as before we conclude that $P*K(\aaa)=0$. Otherwise, since $\mathbb S_\aaa \subset V_1$, using Proposition \ref{prodstar}, we have
		\[P*K(\aaa)=P(\aaa)K(P(\aaa)^{-1}\aaa P(\aaa))=0.\] 
		Thus $P*K \in \mathcal J(\mathcal V_c(I))=I$ since $I$ is radical. \\
	Since $I$ is quasi prime and $P$ does not belong to $I$, for any $K \in \mathcal J(V_1)$, $K^s$ belongs to $I$, that is $\mathcal S(\mathcal J(V_1))\subseteq I$. Since $I$ is radical we have that
	$\sqrt{\mathcal S(\mathcal J(V_1))}\subseteq I$. Moreover, recalling Corollary \ref{essealg}, 
	\[ \sqrt{\mathcal S(\mathcal J(V_1))}= \mathcal J (\mathbb S_{\mathcal V_c (\mathcal J(V_1))})= \mathcal J ( \mathcal V_c (\mathcal J(V_1)))=\mathcal J(V_1),\]
	 where we used the fact that $\mathcal V_c(\mathcal J(V_1))$ has spherical symmetry and $\mathcal J(V_1)$ is radical.
Therefore, we deduce that $\mathcal J(V_1)\subseteq I$. The other inclusion holds since $V_1 \subseteq \mathcal V_c(I)$, and thus $I=\mathcal J(V_1)$.
Using Proposition \ref{Zc} we get that $I$ is a two-sided ideal and then, thanks to Proposition \ref{twosided},  $\mathcal V_c(I)=\mathbb S_{\mathcal V_c(I)}$ which is irreducible.  
	\end{proof}
}

\noindent Consequently, we also have
\begin{corollary}\label{coro3}
	If $I$ is a {quasi} prime {and radical} ideal of $\mathbb H[q_1,\ldots,q_n]$ such that $\sqrt{\mathcal S(I)}$ is quasi prime, then $\mathcal V_c(I)$ is irreducible.
\end{corollary}
\begin{proof}
{The ideal $\sqrt{\mathcal S(I)}$ is by assumption quasi prime and radical. Moreover, recalling that $
	\mathcal V_c(\sqrt{\mathcal S(I)})=\mathcal V_c({\mathcal S(I)})$, we have that
		\[\mathbb S_{\mathcal V_c(\sqrt{\mathcal S(I)}) }=\mathbb S_{\mathcal V_c({\mathcal S(I)}) }=\mathcal V_c({\mathcal S(I)})\]
		since $\mathcal V_c ({\mathcal S(I)})$ has always spherical symmetry. 
		Thus $\mathcal V_c(\sqrt{\mathcal S(I)})$ has spherical symmetry as well.
Corollary \ref{coro} then yields that $\mathcal V_c (\sqrt{\mathcal S(I)})$ is irreducible. Since $\mathcal V_c (\sqrt{\mathcal S(I)})=\mathcal V_c({\mathcal S(I)})=\mathbb S_{\mathcal V_c(I)}$,  
Theorem \ref{cpirred} allows us to conclude.
}
\end{proof}
	
	\subsection{Irreducibility of ideals} 

\noindent	In Algebra there exists a classical notion of irreducibility for ideals, {which in our framework reads as follows}.  
	\begin{definition}\label{redI}
		A right ideal $I$ {of $\mathbb H[q_1,\ldots,q_n]$} is said to be {\em irreducible} if, whenever $I=I_1\cap I_2$ {for two right ideals $I_1,I_2$},  it follows that either $I=I_1$ or $I=I_2$.
		If not, then 
		$I$ is said to be {\em reducible}.
	\end{definition}

\noindent 	In view of the geometric properties of the zeros of quaternionic slice regular polynomials, it is useful to introduce also the following definition. \begin{definition}
			A right ideal $I$ {of $\mathbb H[q_1,\ldots,q_n]$} is said to be {\em s-irreducible} if, whenever $I=I_1\cap \mathcal S(I_2)$ for some right ideals $I_1,I_2$,  it follows that either $I=I_1$ or $I=\mathcal S(I_2)$.
			If not, then 
			$I$ is said to be {\em s-reducible}.
		\end{definition}
\noindent 	The slice algebraic set associated with an irreducible radical ideal is irreducible.  
	\begin{proposition}\label{idirred}
		Let $I$ be a right ideal {of $\mathbb{H}[q_1,\ldots,q_n]$} whose radical $\sqrt {I}$ is an irreducible  right ideal. Then $\mathcal V_c(I)$ is an irreducible {slice} algebraic set.
		
	\end{proposition}
	
	\begin{proof}
		Suppose by contradiction that
		\[\mathcal{V}_c(I)=\mathcal{V}_c(I_1)\cup \mathcal{V}_c(I_2)\]
		is reducible. Hence, thanks to the Strong Nullstellensatz \ref{snull} and to Proposition \ref{vanishing},
		\[\sqrt{I}=\mathcal{J}(\mathcal{V}_c(I))=\mathcal{J}(\mathcal{V}_c(I_1)\cup \mathcal{V}_c(I_2))= \mathcal{J}(\mathcal{V}_c(I_1))\cap  \mathcal{J}(\mathcal{V}_c(I_2))=\sqrt{I_1}\cap \sqrt{I_2},\]
		that is $\sqrt{I}$ is reducible.
	\end{proof}
	
		\noindent The converse is not true in general. 
	\begin{example}\label{qp_not_irred}
		Let $I=\langle q^2+1\rangle$. Then, $\mathcal V_c(I)=\mathbb S_i$ is irreducible, but the ideal $I$ is radical and reducible, as $I=\langle q-j \rangle \cap \langle q-k \rangle$.
	\end{example}

\noindent 	Thanks to Theorem \ref{VirredIprime} and Proposition \ref{idirred}, we have
	\begin{corollary}
		If $I$ is a radical and irreducible right ideal {of $\mathbb H[q_1,\ldots,q_n]$}, then $I$ is a quasi prime ideal.

	\end{corollary}
\noindent	{In the other direction, we can prove the following}
{\begin{corollary}
		If $I$ is a quasi prime and radical right ideal of $\mathbb H[q_1,\ldots,q_n]$, then $I$ is s-irreducible.
	\end{corollary}
	\begin{proof}
		Suppose by contradiction that $I=I_1\cap \mathcal S(I_2)$. Then, thanks to Corollary \ref{essei}, $\mathcal V_c(I)=\mathcal V_c(I_1)\cup \mathcal V_c(\mathcal S(I_2))$. Which is in contradiction with Theorem \ref{nosfere}. 
	\end{proof}
}
%
%
﻿
﻿
﻿
\section{Principal  ideals}\label{Sprinc}
﻿
\noindent If $I$ is a right ideal of $\mathbb{H}[q_1,\ldots,q_n]$ generated by {the polynomial $H$, then
$\mathcal{V}_c(I)$ coincides with the set $Z_c(H)$ of zeros of $H$.}
﻿
﻿
\begin{proposition} \label{irrcp} Let $H\in \mathbb H[q_1,\ldots,q_n]$. The ideal  $I=\langle H\rangle$ is  completely prime if and only if the polynomial  $H$ is irreducible.\end{proposition}
\begin{proof} Suppose{ that $I$ is completely prime and, by contradiction, that $H=P*Q$, with $Q$ a non-constant irreducible factor, and $P$ non-constant.
 Note that $(P*Q)*Q^c=P*Q^s\subseteq I$ and $Q^s*I=I*Q^s\subseteq
 I$}. Then either $P$ or $Q^s$ belong to $I$. If $P\in I$, then
 $\langle P \rangle \subset I=\langle P*Q \rangle \subset \langle P
 \rangle$, hence $I=\langle P\rangle$ and $Q$ is a constant polynomial, a contradiction. 
 Otherwise, if $Q^s\in
 I$,  we
 have $$Q^s=P*Q*R$$ where $R$ is a polynomial in $\mathbb H[q_1,\ldots, q_n]$. Then, $P=Q^s*R^{-*}*Q^{-*}=R^{-*}*Q^c$, and
 thus $Q^c=R*P$ and $Q=P^c*R^c$. Therefore, taking into account that $Q$ is irreducible, $R^c$ is constant that is $\langle H\rangle=\langle Q^s\rangle$. Now note that a completely prime ideal cannot be generated by the symmetrisation of an irreducible factor. Indeed, if $I=\langle Q^s\rangle$ we have $Q^s=Q*Q^c\in I$ and $Q*I\subseteq I$, thus either $Q\in I$ or $Q^c\in I$, but none of them can be written as a multiple of $Q^s$, {leading to a contradiction.}

Conversely, {suppose that $H$ is irreducible and
take $P$ and $Q$ in $\mathbb{H}[q_1,\ldots, q_n]$  such that $P*Q\in I$ and $P*I\subseteq I$. Therefore there exist $R,S$ polynomials in $\mathbb{H}[q_1,\ldots, q_n]$ such that $$P*Q=H*R$$ $$P*H=H*S.$$
Thus
$P=H*S*H^{-*}$ and $R=S*H^{-*}*Q$. Since $H$ is irreducible, we conclude that either $S=S_1*H$ and then $P=H*S_1 \in I$, or $Q=H*Q_1 \in I$. }
 \end{proof}

﻿\noindent {Specializing Theorem \ref{VirredIprime} to the case of a principal ideal, we get the following stronger result.}
 \begin{theorem} \label{symprime} Let $I=\langle H \rangle$ be a radical right ideal {of $\mathbb{H}[q_1,\ldots, q_n]$} such that $\mathcal{V}_c(I)$ is irreducible. 
Then 
either $H$ is irreducible
or $H$ is the symmetrization of an irreducible polynomial.
 \end{theorem}
 \begin{proof} 
	Suppose $H$ is reducible that is $H=P*Q$, {where }we can assume that $Q$ is a non-constant  irreducible factor. We have that $P*Q$ vanishes on  $\mathcal{V}(I)$, and thus it vanishes also on $\mathcal{V}_c(I).$ On this set $P*Q$ has an {explicit} expression (see Proposition \ref{prodstar}). Indeed for all $\textbf{\em a}\in \mathcal{V}_c(I)$, either $P(\textbf {\em a})=0$, or $$(P*Q)(\textbf{\em a})=P(\textbf{\em a})\cdot Q({P}^{-1}(\textbf{\em a})\textbf {\em a}P(\textbf{\em a})$$
	and $Q({P}^{-1}(\textbf{\em a})\textbf{\em a}P(\textbf{\em
		a})=0$. 
		Take now the two subsets of $\mathcal{V}_c(I),$ given
	by $$V_1=\mathcal{V}(I)\cap \mathcal{V}_c(\langle P\rangle)=\mathcal{V}_c(\langle P\rangle)$$
	and
	
	$$  V_2=\mathcal{V}(I)\cap \mathbb{S}_{\mathcal{V}_c(\langle
		Q\rangle)}=
	\mathcal{V}_c{(I)}\cap\mathcal{V}_c(\langle
	{Q^s}\rangle)=
	\mathcal{V}_c(\langle {Q^s},H\rangle)$$
	Now  either ${ \textbf{\em a}}$ belongs to $V_1$ or one of
	its conjugates belongs to $V_c(\langle Q\rangle)$, thus ${\textbf{\em a}}$
	belongs to $V_2$; therefore $$\mathcal{V}_c(I)= V_1\cup {V_2}$$ 
	
	\noindent Since $\mathcal{V}_c(I)$ is irreducible, then either
	$$\mathcal{V}_c(I)=V_1=\mathcal{V}_c(\langle P\rangle)$$ 
	and thus $P$ vanishes on $\mathcal{V}_c(I)$ that is $P\in \mathcal{J}(\mathcal{V}_c(I))=\sqrt{I}=\langle H\rangle$. {Therefore there exists $P_1 \in \mathbb{H}[q_1,\ldots, q_n]$ such that $P=H*P_1=P*Q*P_1$} that is $Q*P_1=1$ in contradiction with the hypothesis on $Q$.
	
	\noindent The other possibility is that $\mathcal{V}_c(\langle
	H\rangle)=V_2$,
 and hence $Q^s\in \langle H\rangle$.  
﻿
\noindent therefore there exists a polynomial $Q_1$ such that 
	\[H^s=P^s*Q^s=P^s*H*Q_1=H*Q_1*P^s,\]
	 that is \[H^c=Q^c*P^c=Q_1*P^s\]
	 which implies that $Q_1*P=Q^c$. Recalling that $Q^c$ is irreducible (since $Q$ is irreducible), we get that $Q_1$ is a constant and thus $H$ is a multiple of $Q^s$.
	i.e. $H$ is the
	symmetrization of an irreducible factor.
\end{proof}
﻿
%
  \begin{proposition}\label{Hsquasiprime} If $H$ is an irreducible polynomial such that $H\neq H^c$, then $\langle H^s \rangle$ is a quasi prime and radical ideal. 
 \end{proposition}
\begin{proof}
	Let $P,Q$ be such that $P*Q=H^s*R$. If $P \notin \langle H^s \rangle$, then:
	either $Q \in \langle H^s \rangle$ and hence $Q^s \in \langle H^s \rangle$; or there exists $\hat H$ such that
	\[P=P_1*\hat H \quad \text{and} \quad Q=\hat H^c * Q_1, \quad \text{with $\hat H^s=H^s$.}\]
	Thus $Q^s \in \langle H^s \rangle$ {and $\langle H^s \rangle$ is a quasi prime ideal.}
	
{\noindent Consider now} $P\in \sqrt{\langle H^s \rangle}=\mathcal J(\mathcal V_c(\langle H^s \rangle))$. Then $P$ vanishes on $\mathbb S_{\mathcal V_c(H)}$ and thus on $\mathcal V_c(H) \cup \mathcal V_c(H^c)$. Since $H$ is irreducible, we have that both $\langle H\rangle $ and $\langle H^c \rangle$ are radical and hence $P \in \langle H\rangle \cap \langle H^c \rangle$. 	Therefore there exist $P_1,P_2$ such that
\[P=H*P_1=H^c*P_2.\]
Since $H\neq H^c$ {and are both irreducible}, then necessarily there exists $P_3$ such that $P_2=H*P_3$. Thus $P \in \langle H^s \rangle$ {and $\langle H^s \rangle$ is a radical ideal.}
\end{proof}
\noindent Observe that if $H=H^c$, then $\langle H^s \rangle=\langle H^{*2} \rangle$ which is not a radical ideal.
﻿
%
﻿
﻿
{
	
\noindent For principal ideals, the following theorem provides equivalent conditions for  the irreducibility of the corresponding {slice} algebraic set.
﻿
﻿
﻿
\begin{theorem}\label{principal}
	Consider  $I=\langle H \rangle$ with
		$H\in \mathbb H[q_1,\ldots q_n]$. The following are equivalent:
	\begin{enumerate}
		\item the ideal $I$ is quasi prime;
		\item $H$ is an irreducible polynomial or $H=Q^s$ with $Q$
		an irreducible polynomial $Q\neq Q^c$;
		\item $I$ is a radical ideal and $\mathcal V_c(I)$ is an irreducible slice algebraic set.
	
	\end{enumerate}
	\end{theorem}
\begin{proof}
\noindent{ \bf $(1) \Rightarrow (2)$} Suppose that $H=P*Q$, with  $Q$ a non-constant irreducible factor, and $P$ a non-constant polynomial.
Since $I$ is quasi prime, either $P$ or $Q^s$ belong to $I$. If $P\in I$, then
$\langle P \rangle \subset I=\langle P*Q \rangle \subset \langle P
\rangle$, hence $I=\langle P\rangle$ and $Q$ is a constant, a contradiction. 
Otherwise, if $Q^s\in
I$, we
have $$Q^s=P*Q*R$$ with $R \in \mathbb H[q_1,\ldots,q_n]$. Then, proceeding as in the proof of Proposition \ref{irrcp},  taking into account the irreducibility of $Q$, we get that $P$ is a multiple of $Q^c$ and thus
$H=Q^s$ up to a multiplicative constant.\\
\noindent{ \bf $(2) \Rightarrow (3)$}
Assume that 
$I$ is a principal  ideal generated by an irreducible polynomial $H\in\mathbb{H}[q_1,\ldots,q_n]$. Then, from Proposition \ref{irrcp}, it follows that $I$ is completely prime and hence {also} radical.
Combining Corollary \ref{coro} and Proposition \ref{Hsquasiprime} we get 
$\mathcal{V}_c(\langle H^s \rangle)$ is irreducible.
 We now apply Theorem \ref{cpirred} and conclude that $\mathcal{V}_c(I)$ is irreducible. 
 If $I$ is a principal  ideal generated by $H=Q^s$, with 
$Q$ an irreducible polynomial such that $Q\neq Q^c$, then $I$ is quasi prime and radical, thanks to  
 Proposition \ref{Hsquasiprime}. Since  $\mathcal{V}_c(\langle Q^s \rangle)$ has spherical symmetry, from Corollary \ref{coro}, we get that  $\mathcal{V}_c(I)$ is irreducible. 
 \\
\noindent{ \bf $(3) \Rightarrow (1)$} This follows directly from Theorem \ref{VirredIprime}.
\end{proof}

\noindent {{Thanks to the previous result we can exhibit an explicit example of an irreducible slice algebraic set which is the union of an infinite number of spherical sets. 
	\begin{example} Consider the right ideal $I$ generated by the polynomial $P(q_1,q_2)=q_1q_2+1$. Since $P$ is irreducible, we have that 
		$\mathcal{V}_c(I)=\bigcup_{a\in \mathbb{H}}\mathbb{S}_{(a,-a^{-1})}$ is irreducible.
	\end{example}}
}
}
﻿
﻿
%
 
﻿
﻿
﻿
﻿
﻿
﻿
﻿
﻿
﻿
﻿
﻿
﻿
﻿
﻿
﻿
﻿
﻿

 \end{document}